\documentclass[]{interact}

\usepackage{epstopdf}
\usepackage[caption=false]{subfig}

\usepackage[numbers,sort&compress]{natbib}
\bibpunct[, ]{[}{]}{,}{n}{,}{,}

\theoremstyle{plain}
\newtheorem{theorem}{Theorem}[section]
\newtheorem{lemma}[theorem]{Lemma}
\newtheorem{corollary}[theorem]{Corollary}

\theoremstyle{definition}
\newtheorem{definition}[theorem]{Definition}

\theoremstyle{remark}

\begin{document}

\title{Stability Analysis for a Class of Sparse  Optimization Problems}
\author{
\name{Jialiang Xu\textsuperscript{a}\thanks{Yun-Bin Zhao. Email: y.zhao.2@bham.ac.uk} and Yun-Bin Zhao\textsuperscript{b}}
\affil{\textsuperscript{a,b} School of Mathematics, University of Birmingham, Edgbaston, Birmingham, B15 2TT, United Kingdom}
}
\maketitle

\begin{abstract}  The sparse optimization problems  arise in many areas of science and engineering,  such as compressed  sensing, image processing, statistical  and machine learning. The $\ell_{0}$-minimization problem is one of such optimization problems, which is typically used to deal with signal recovery. The $\ell_{1}$-minimization method is  one of the plausible approaches for solving the $\ell_{0}$-minimization problems, and thus the stability of such a numerical method is vital for signal recovery. In this paper,   we establish a stability result for the  $\ell_{1}$-minimization problems associated with a general class of $\ell_{0}$-minimization problems. To this goal,  we introduce the concept of  restricted weak range space property (RSP) of a transposed sensing matrix, which is a generalized version of the weak RSP of the transposed sensing matrix introduced in [Zhao et al., Math. Oper. Res., 44(2019), 175-193].  The stability result  established in this paper includes  several  existing  ones as special cases.
 
\end{abstract}


\section{Introduction}
 The sparsity   is  a  useful assumption under which  the sparse optimization models arise frequently in many areas in science and engineering. Let $A\in R^{m\times n} (m\ll n) $, $B\in R^{l\times n} (l< n)$ and $U\in R^{m\times h} (m\ll h)$ be three given full-row-rank matrices. Let  $y\in R^{m}$ and $b\in R^{l}$ be given vectors  and $\varepsilon$ be a positive number.   Consider the following sparse optimization model:
\begin{equation}\label{Ps0}
\begin{array}{lcl}
&\min\limits_{x\in R^{n}}&~\left \| x \right \|_{0}\\
& $s.t$.&~ a_{1}\left \| y-Ax \right \|_{2}+a_{2}\left\| U^{T}(Ax-y)\right\|_{\infty}+a_{3}\left\| U^{T}(Ax-y)\right\|_{1} \leq \varepsilon \\
&&~ Bx\leq b,
\end{array}
\end{equation}
where $\Vert x\Vert_{0}$ is called the `$\ell_{0}$-norm' which counts the number of nonzero components of $x$, and $a_{1}, a_{2}$ and $a_{3}$ are  given nonnegative parameters satisfying  $\sum_{i=1}^{3} a_{i}=1$. Many problems in   signal and image processing (see, e.g., \cite{C06,D06,Redbook}) and statistical regressions \cite{greenbook} can be formulated as  the form \eqref{Ps0} or its special cases.
In problem \eqref{Ps0}, the constraint $Bx\leq b$  is motivated by some practical applications. For instance, many signal recovery models  might  need to include certain  constraints   reflecting   special structures of the target signal.  For simplicity,  we define $$\phi(x)=U^{T}(Ax-y),$$
and write the problem \eqref{Ps0} as
\begin{equation*}
\min\limits_{x\in R^{n}}\left\{ \left \| x \right \|_{0}: ~ a_{1}\left \| y-Ax \right \|_{2}+a_{2}\left\| \phi(x)\right\|_{\infty}+a_{3}\left\| \phi(x)\right\|_{1} \leq \varepsilon, Bx\leq b \right\}.
\end{equation*}
The following  $\ell_{0}$-minimization  models   are clearly  the special cases of \eqref{Ps0}:
 \begin{equation*}
\begin{array}{ll}
$(C1)$~\min\limits_{x} \lbrace \Vert x\Vert_{0}:~ y=Ax\rbrace; & $(C2)$~ \min\limits_{x} \lbrace  \Vert x\Vert_{0}:~\left\| y-Ax\right\|_{2}\leq \varepsilon\rbrace;\\
$(C3)$ ~\min\limits_{x} \lbrace  \Vert x\Vert_{0}:~\left\| U^{T}(Ax-y)\right\|_{1}\leq \varepsilon\rbrace; &  $(C4)$ ~\min\limits_{x} \lbrace \Vert x\Vert_{0}:~\left\| U^{T}(Ax-y)\right\|_{\infty}\leq \varepsilon\rbrace.
\end{array}
\end{equation*}
The problem (C1) is often called the standard $\ell_{0}$-minimization problem \cite{Redbook,candes2005,zhaobook2018}.
Two structured sparsity models, called  the nonnegative sparsity model  \cite{candes2005,CERT2006,Redbook,zhaobook2018} and the monotonic sparsity model  (isotonic regression) \cite{THT2011, greenbook},  are also  the special cases of the model \eqref{Ps0}.

 It is well known that  $\ell_{1}$-minimization  is a useful method to solve the  $\ell_{0}$-minimization problem. By replacing the $\ell_{0}$-norm with the $\ell_{1}$-norm in problem \eqref{Ps0},  we immediately obtain the $\ell_{1}$-minimization problem
\begin{equation}\label{Ps}
\min\limits_{x}\lbrace\left \| x \right \|_{1}: ~a_{1}\left \| y-Ax \right \|_{2}+a_{2}\left\| \phi(x)\right\|_{\infty}+a_{3}\left\| \phi(x)\right\|_{1} \leq \varepsilon, Bx\leq b \rbrace.
\end{equation}
Similar to its $\ell_{0}$ counterpart, the problem \eqref{Ps} includes the following special cases:
 \begin{equation*}
\begin{array}{ll}
$(D1)$~\min\limits_{x} \lbrace \Vert x\Vert_{1}:~ y=Ax\rbrace; & $(D2)$~ \min\limits_{x} \lbrace\Vert x\Vert_{1}:~\left\| y-Ax\right\|_{2}\leq \varepsilon\rbrace;\\
$(D3)$ ~\min\limits_{x} \lbrace \Vert x\Vert_{1}:~\left\| U^{T}(Ax-y)\right\|_{1}\leq \varepsilon\rbrace; &  $(D4)$ ~\min\limits_{x} \lbrace  \Vert x\Vert_{1}:~\left\| U^{T}(Ax-y)\right\|_{\infty}\leq \varepsilon\rbrace.
\end{array}
\end{equation*}
 The problem   (D2)  is often called  quadratically constrained basis pursuit \cite{Redbook,CP2011,zhaobook2018}, and it reduces to (D1) if $\varepsilon=0$,  which is called  standard $\ell_{1}$-minimization or the basis pursuit \cite{cohen2009,candes2005,donoho2001,zhao2013, Redbook}.  The  problem  (D4) is the type of   Dantzig Selectors  \cite{candes2007,Redbook}.

From both numerical and theoretical viewpoints,  it is important to know  how close  the solutions of $\ell_{0}$- and $\ell_{1}$-minimization problems are. To address this question, one needs to study  the stability of $\ell_{1}$-minimization methods.  The stability of a sparse optimization method can be described as follows:    For  any $ x\in R^{n}$  in the feasible set of  a sparse optimization problem,  the solution  $x^{\#}$ generated by the method    satisfies the following bound:
\begin{equation}\label{criterion3}
 \small\left\| x-x^{\#}\small\right\|_{2}\leq C_{1}\sigma_{k}(x)_{1}+C_{2}\varepsilon
\end{equation}
 where  $C_{1}$ and $C_{2}$ are constants, and $\sigma_{k}(x)_{1}$ is called the  error of the  best $k$-term approximation of the vector $x$ (see, e.g., \cite{cohen2009,Redbook}):
\begin{equation*}\label{k-best}
\sigma _{k}\left ( x \right )_{1}=\min_{z}\lbrace \left \| x-z \right \|_{1}:~\Vert z\Vert_{0}\leq k\rbrace.
\end{equation*}

 In this paper, we  establish a stability result for the $\ell_{1}$-minimization method \eqref{Ps}.  The stability of  (D1) and (D2) has been investigated by  Donoho, Cand\`es, Tao, Romberg and others \cite{D03,D06,C06,CERT2006,candes2005,cohen2009,zhang2005,eldarbook2012,CCW2016} under various assumptions such as  the  so-called  restricted isometry property (RIP) of order $k$, mutual coherence, stable null space property (NSP) of order $k$ or robust NSP of order $k$.  The  RIP  of order $k$  was introduced by Cand\`es and Tao \cite{candes2005} to study the stability of $\ell_{1}$-minimization.
The singular-value-property-based stability analysis for (D1), (D2) and the Dantzig Selector  have  also been performed by Tang and Nehorai in \cite{tang2011}.

 A new and unified stability analysis for $\ell_{1}$-minimization methods has been developed by Zhao, Jiang and Luo \cite{YHZ2018}  under the assumption of  weak RSP of order $k$, which has been proven as a  necessary and sufficient condition for  the standard $\ell_{1}$-minimization to be stable.   The main differnece between the weak-RSP-based-analysis and existing ones lies in  the constants $C_{1}$ and $C_{2}$ in \eqref{criterion3}. Specifically,  the constants $ C_{1}$ and $C_{2}$ in  \eqref{criterion3} are determined by  the RIP or NSP constant in existing analysis  \cite{Redbook,CCW2016,candes2005}. However, in \cite{YHZ2018,zhaobook2018}, these constants are determined by the so-called Robinson's constant. Motivated by the new analysis tool introduced in \cite{YHZ2018},   we  develop the stability result  for the  model \eqref{Ps} in this paper under the assumption of restricted weak range space property ($\mathrm{RSP}$) of order $k$ (which will be introduced in next section). Our result extends the stability theorem for $\ell_{1}$-minimization established by Zhao et al. \cite{YHZ2018,zhaolistability,zhaobook2018}.

This paper is organized as follows.  In Section \ref{section weak rsp}, we  introduce the concept of  restricted weak $\mathrm{RSP}$ of order $k$.  An approximation of the solution set of \eqref{Ps}  will be discussed in Section \ref{section approxiamtion problem}.  Then, in Section \ref{section stability theorem}, we show the main stability result of this paper.  Finally,   some special cases are discussed in Section \ref{section special case}.

\section*{Notation}
 The field of real numbers is denoted by $R$ and the $n$-dimensional Euclidean space is denoted by $R^{n}$. Let $R^{n}_{+}$ and $R^{n}_{-}$ be the sets of nonnegative and nonpositive vectors, respectively.  Unless otherwise stated, the identity matrix of suitable size is denoted by $I$.   Given a vector $u\in R^{n}$,  $\vert u\vert$, $(u)^{+}$ and $(u)^{-}$ denote the vectors with components $\vert u\vert_{j}=\vert u_{j}\vert$, $[(u)^{+}]_{j}=\max\lbrace u_{j},0\rbrace$ and $[(u)^{-}]_{j}=\min\lbrace u_{j},0\rbrace$, $j=1,...,n$, respectively. The cardinality of the set $S$ is denoted by $\vert S\vert$ and the complementary set of $S\subseteq \left\{ 1,...,n\right\}$ is denoted by $\bar{S}$, i.e., $\bar{S}=\lbrace 1,...,n\rbrace \setminus S$.  For a given vector $x\in R^{n}$,    $x_{S}$ denotes the vector supported on $S$.   $a_{i,j}$ denotes the entry of the matrix $A$ in row $i$ and column $j$.   For the  set $S\subseteq \lbrace 1,...,n\rbrace$, $A_{S}$ denotes the submatrix of $A\in R^{m\times n}$ obtained by deleting the columns indexed by $\bar{S}$. For a matrix $A=\left(a_{i,j}\right)$, $\vert A\vert$ represents the absolute version of $A$, i.e., $\vert A\vert=\left(\vert a_{i,j}\vert\right)$. $\mathcal{R}\left(A^{T}\right)=\lbrace A^{T}y: y\in R^{m}\rbrace$ is the range space of $A^{T}$.    $\left\| x\right\|_{p}=\left(\sum_{i=1}^{n}\left| x_{i}\right|^{p}\right)^{1/p}$, where $p\geq 1$,   is a norm, called the $\ell_{p}$-norm of $x$.  $\left \| x \right \|_{\infty}=\max_{i=1}^{n}\left| x_{i}\right|$ is called the $\ell_{\infty}$-norm of $x$. For $1\leq p,q\leq \infty$,
$\left\| A\right\|_{p\rightarrow q}=\sup_{\left\| x\right\|_{p}\leq 1} \left\| Ax\right\|_{q}$ is the matrix norm induced by $\ell_{p}$- and $\ell_{q}$-norms.

\section{Restricted weak range space property}\label{section weak rsp}
The  $\mathrm{RSP}$ of order $k$ of a  transposed matrix was first introduced  in \cite{zhao2013,zhao2014} to develop   a  necessary and sufficient condition for the uniform recovery of sparse signals via $\ell_{1}$-minimization.  
 Zhao et al. \cite{YHZ2018} generalised the  $\mathrm{RSP}$ of order $k$ to the following weak $\mathrm{RSP}$ of order $k$ to develop a stability theory for convex optimization algorithms:
 \begin{definition}[weak $\mathrm{RSP}$ of order $k$]
Given a matrix $A\in R^{m\times n}$,  $A^{T}$ is said to satisfy the weak $\mathrm{RSP}$ order $k$ if for any two disjoint sets $J_{1}, J_{2}\subseteq \lbrace 1,...,n\rbrace$ satisfying $\vert J_{1}\vert+\vert J_{2}\vert\leq k$, there exists a vector $\eta\in \mathcal{R}\left(A^{T}\right)$ such that
\begin{equation*}
\left\{\begin{array}{lll}
  \eta_{i}=1 & \mathrm{if} ~i\in J_{1},\\
 \eta_{i}=-1  & \mathrm{if} ~i\in J_{2},\\
 \vert \eta_{i}\vert\leq 1  & \mathrm{if} ~i\notin J_{1} \cup J_{2}.
\end{array}\right.
\end{equation*}
\end{definition}
\noindent In \cite{YHZ2018,zhaobook2018}, it was shown that the weak $\mathrm{RSP}$ of order $k$  is a sufficient condition for the stability of many convex optimization methods, and it is also a  necessary stability condition for many optimization methods.

Different from the problems (D1)-(D4), the problem \eqref{Ps} is  more general  than these models. To investigate the stability of the problem \eqref{Ps}, we need to extend the  notion of weak RSP of order $k$ to the so-called restricted weak $\mathrm{RSP}$ of order $k$, which  is defined as follows:
  \begin{definition}[Restricted weak $\mathrm{RSP}$ of order $k$]
Given  matrices $A\in R^{m\times n}$ and $B\in R^{l\times n}$, the pair $\left(A^{T}, B^{T}\right)$ is said to satisfy the restricted weak $\mathrm{RSP}$ of order $k$ if for any two disjoint sets  $J_{1}, J_{2}\subseteq \lbrace 1,...,n\rbrace$ satisfying $\vert J_{1}\vert+\vert J_{2}\vert\leq k$, there exists a vector $\eta\in \mathcal{R}\left(A^{T}, B^{T}\right)$ such that $\eta=\left(A^{T},B^{T}\right)\left(\begin{array}{c}
\nu\\h
\end{array} \right)$  where $\nu\in R^{m}$, $h\in R^{l}_{-}$ and
\begin{equation*}
\left\{\begin{array}{ll}
 \eta_{i}=1 & \mathrm{if} ~ i\in J_{1},\\
 \eta_{i}=-1 &\mathrm{if} ~ i\in J_{2},\\
 \vert \eta_{i}\vert\leq 1 &\mathrm{if} ~ i\notin J_{1} \cup J_{2}.
\end{array}\right.
\end{equation*}
\end{definition}
\noindent  It is worth mentioning that a generalized version of the RSP  of order $k$ is also used in  \cite{zhaochun2016} to study  the exact sign recovery in 1-bit compressive sensing.

\section{Approximation of \eqref{Ps} and its solution set} \label{section approxiamtion problem}
By introducing the slack variables $r$, $s$, $\xi$ and $v$, the problem \eqref{Ps} can be rewritten as
\begin{equation}\label{Ps1}
\begin{array}{lcl}
&\min\limits_{(x, r, s, \xi, v)}&\left \| x \right \|_{1}\\
& $s.t$.& a_{1}s+a_{2}\xi+a_{3}\left(e^{h}\right)^{T}v \leq \varepsilon,\\
& & r\in s \mathcal{B}, ~r=y-Ax,~(s,\xi, v)\geq 0,\\
& & \left\| \phi(x)\right\|_{\infty}\leq \xi, ~\left|  \phi(x)\right| \leq v,~Bx\leq b,
\end{array}
\end{equation}
where $e^{h}$ is the vector of ones in $R^{h}$ and $\mathcal{B}$ is the unit $\ell_{2}$-ball  defined as  $\mathcal{B}=\lbrace z\in R^{m}:\left\|z \right\|_{2}\leq 1 \rbrace$. The unit ball $\mathcal{B}$ can be also  described as
\begin{equation}\label{unit ball}
\mathcal{B}=\bigcap\limits_{\Vert \textbf{a}\Vert_{2}=1}\lbrace z\in R^{m}: \textbf{a}^{T}z\leq 1\rbrace.
\end{equation}
 Denote the set ${\rm E}$ by
\begin{equation*}
\begin{split}
{\rm E}=\lbrace (x, s, \xi,v): &~ a_{1}s+a_{2}\xi+a_{3}(e^{h})^{T}v \leq \varepsilon,~ Bx\leq b,~\left\| \phi(x)\right\|_{\infty}\leq \xi,\\
& ~\left|  \phi(x)\right| \leq v,~(s,\xi, v)\geq 0\rbrace,
\end{split}
\end{equation*}
and hence the solution set of \eqref{Ps1} can be represented as
\begin{equation}\label{Omega}
\Omega^{*}=\lbrace (x,r,s,\xi,v): \Vert x\Vert_{1}\leq \theta^{*},~r\in s \mathcal{B}, ~r=y-Ax,~ (x,s,\xi,v)\in {\rm E} \rbrace,
\end{equation}
where $\theta^{*}$ is the optimal value of \eqref{Ps1}.  By replacing $\mathcal{B}$ in \eqref{Omega} with  a polytope $P\supseteq \mathcal{B}$, we can get the relaxation  of  $\Omega^{*}$, denoted by $\Omega_{ P}$, i.e.,
\begin{equation}\label{Omega1}
\Omega_{ P}=\lbrace (x,r,s,\xi,v): \Vert x\Vert_{1}\leq \theta^{*},~r\in s P, ~r=y-Ax,~ (x,s,\xi,v)\in {\rm E} \rbrace.
\end{equation}
The polytope  $\Omega_{P}$ can   approximate $\Omega^{*}$ to any  level of accuracy provided that $P$ is chosen suitably.
Recall the Hausdorff metric of  two sets $M_{1}, M_{2}\subseteq  R^{m}$:
\begin{equation*}\label{hausdorff}
\delta^{\mathcal{H}}(M_{1}, M_{2})=\max\biggr \lbrace\sup_{x\in M_{1}} \inf_{z\in M_{2}}\left\| x-z\right\|_{2}, \sup_{z\in M_{2}} \inf_{x\in M_{1}}\left\| x-z\right\|_{2}\biggr \rbrace.
\end{equation*}
Following the analysis  in \cite{YHZ2018, zhaobook2018} (see Lemmas 5.1, 5.2 and 5.3  in \cite{YHZ2018}), we can obtain  the following lemma:

\begin{lemma}\label{lemma approximat level}
Let $\varepsilon$ be the given number in problem \eqref{Ps}. Then for any $\varepsilon'\leq \varepsilon$, there exists a polytope approximation $P$ of $\mathcal{B}$ satisfying $P\supseteq \mathcal{B}$ and
\begin{equation}\label{final polytope}
\delta^{\mathcal{H}}\left(\Omega^{*},\Omega_{P}\right)\leq \varepsilon'.
\end{equation}
\end{lemma}

In the remainder of this paper, we fix $\varepsilon'\in (0, \varepsilon]$ and choose the  polytope $P$  such that  $\Omega_{P}$ and $\Omega^{*}$ satisfy \eqref{final polytope}. The polytope $P$ can be represented as the intersection of a finite number of  half spaces:
$$P=\lbrace z\in R^{m}:\left(\textbf{a}^{i}\right)^{T}z\leq 1,  1\leq i\leq L\rbrace,$$
where $\textbf{a}^{i},~1\leq i\leq L$ are some unit vectors (i.e., $\Vert \textbf{a}^{i}\Vert_{2}=1$), and $L$ is an integer number. By adding the $2m$ half spaces $$\left(\beta^{j}\right)^{T}z\leq 1,~ -\left(\beta^{j}\right)^{T}z\leq 1, ~j=1,...,m$$ to $P$, where $\beta^{j}$ is the $j$th column of the $m\times m$  identity matrix,  we obtain the following polytope:
\begin{equation}\label{polytope appximated by B}
\begin{split}
P_{0}&=P\cap\left\{z\in R^{m}: \left(\beta^{j}\right)^{T}z\leq 1, -\left(\beta^{j}\right)^{T}z\leq 1, j=1,...,m\right\}\\
&=\left\{z\in R^{m}: \left(\textbf{a}^{i}\right)^{T}z \leq 1,  1\leq i\leq L;\left(\beta^{j}\right)^{T}z\leq 1, -\left(\beta^{j}\right)^{T}z\leq 1, j=1,...,m\right\}.
\end{split}
\end{equation}
 We define $T$ as the collection of the  vectors $\textbf{a}^{i}$ and $\pm\beta^{j}$ in $P_{0}$, that is,
\begin{equation*}\label{collection}
T:=\lbrace \textbf{a}^{i}: 1\leq i\leq L\rbrace \cup \lbrace \pm \beta^{j}: 1\leq j\leq m\rbrace.
\end{equation*}
Clearly, $P_{0}$ still satisfies  \eqref{final polytope} in Lemma \ref{lemma approximat level}, i.e.,  $$\delta^{\mathcal{H}}\left(\Omega^{*},\Omega_{P_0}\right)\leq \varepsilon'. $$ In the remainder of the chapter, we use the above defined polytope $P_{0}.$  Let  $N=\vert T\vert,$ and let $M_{P_{0}}$ be the matrix with  column vectors in $T$.  Thus  $P_{0}$ can  be written as
 \begin{equation*}\label{polytope approximate B1}
 P_{0}=\lbrace
 z\in R^{m}: \left(M_{P_{0}}\right)^{T} z\leq e^{ N}\rbrace,
\end{equation*}
where $e^{N}$ is the vector of ones in $R^{N}$.

By replacing  $\mathcal{B}$ by $P_{0}$, we  obtain the following approximation of the optimal value $\theta^{*}$ of \eqref{Ps}:
\begin{equation*}\label{optimal value appro}
\begin{split}
\theta^{*}_{P_{0}}:&=\min_{(x,r,s,\xi,v)}\lbrace\left\| x\right\|_{1}: ~r\in s P_{0}, ~r=y-Ax,~ (x,s,\xi,v)\in {\rm E} \rbrace\\
&=\min_{(x,s,\xi,v)}\lbrace\left\| x\right\|_{1}: ~ \left(M_{P_{0}}\right)^{T}(y-Ax)\leq se^{N},~(x,s,\xi,v)\in {\rm E} \rbrace.\\
\end{split}
\end{equation*}
The associated approximation problem of \eqref{Ps} can be written as
 \begin{equation}\label{Ps2}
\min\limits_{(x,  s, \xi, v)}\lbrace \left \| x \right \|_{1}:
 \left(M_{P_{0}}\right)^{T}(y-Ax)\leq se^{N},~ (x,s,\xi,v)\in {\rm E} \rbrace.
\end{equation}
The  solution set of \eqref{Ps2} is
\begin{equation}\label{solution set of appro}
\Omega^{*}_{P_{0}}=\lbrace x\in R^{n}:~\left\| x\right\|_{1}\leq \theta^{*}_{P_{0}}, ~r\in s P_{0}, ~r=y-Ax,~ (x,s,\xi,v)\in {\rm E}\rbrace.
\end{equation}
Note that $\mathcal{B}\subseteq P_{0}$ implies that $\theta^{*}\geq \theta^{*}_{P_{0}}$. So we can see that  $\Omega^{*}_{P_{0}}\subseteq \Omega_{P_{0}}$.  By the definition of  $P_{0}$, we also have $\Omega^{*}\subseteq\Omega_{P_{0}}$. In the next section, we prove the main  result for the problem \eqref{Ps}.

\section{Main result}\label{section stability theorem}
Introducing a variable $t$ yields the following equivalent form of \eqref{Ps2}:
\begin{equation}\label{Ps3}
\begin{array}{lcl}
&\min\limits_{(x, t, s, \xi, v)}& e^{T}t \\
& $s.t$.& a_{1}s+a_{2}\xi+a_{3}\left(e^{h}\right)^{T}v \leq \varepsilon,~ Bx\leq b,~\vert x\vert\leq t,\\
& & \left(M_{P_{0}}\right)^{T}(y-Ax)\leq se^{N},~(t, s,\xi, v)\geq 0,\\
& & \left\| \phi(x)\right\|_{\infty}\leq \xi, \left|  \phi(x)\right| \leq v.\\
\end{array}
\end{equation}
The solution set of \eqref{Ps3} is given as \eqref{solution set of appro}. Note that the above optimization problem is equivalent to a linear programming problem. In fact, the constraint $\left\|\phi(x)\right\|_{\infty}\leq \xi$ can be rewritten as
$\left|  \phi(x)\right|\leq \xi e^{h},$
where $e^{h}$ is the vector of ones in $R^{h}$.
Thus the model \eqref{Ps3} can be rewritten explicitly as the linear programming problem
\begin{equation}\label{Ps4}
\begin{array}{cl}
\min\limits_{(x, t, s, \xi, v)}& e^{T}t \\
  $s.t$.& x+t\geq 0,~-x+t\geq 0,\\
& -a_{1}s-a_{2}\xi-a_{3}\left(e^{h}\right)^{T}v \geq -\varepsilon, M_{P_{0}}^{T}Ax+e^{N}s\geq  M_{P_{0}}^{T}y,\\
& U^{T}Ax+\xi e^{h}\geq U^{T}y, -U^{T}Ax+\xi e^{h}\geq -U^{T}y ,\\
& U^{T}Ax+v\geq U^{T}y, -U^{T}Ax+v\geq -U^{T}y ,\\
& ~ -Bx\geq  -b,~ (t,s,\xi,v)\geq 0.
\end{array}
\end{equation}
The dual problem of \eqref{Ps4} is given as follows:
\begin{equation}\label{Psdual}
\begin{array}{cl}
\max\limits_{w}& -\varepsilon w_{3}+y^{T}M_{P_{0}}w_{4}+y^{T}U(w_{5}-w_{6}+w_{7}-w_{8})-b^{T}w_{9} \\
 $s.t$.& w_{1}-w_{2}+A^{T}M_{P_{0}}w_{4}+A^{T}U(w_{5}-w_{6}+w_{7}-w_{8})-B^{T}w_{9}=0,\\
& w_{1}+w_{2}\leq e,\\
& -a_{1}w_{3}+(e^{N})^{T}w_{4}\leq 0,\\
& -a_{2}w_{3}+(e^{h})^{T}(w_{5}+w_{6})\leq 0,\\
&-a_{3}w_{3}e^{h}+w_{7}+w_{8}\leq 0,\\
& w_{1},w_{2}\in R^{n}_{+},~w_{3}\in R_{+},~w_{4}\in R^{N}_{+}, ~w_{5-8}\in R^{h}_{+},~w_{9}\in R_{+}^{l}.
\end{array}
\end{equation}
 The optimality condition  yields  the following lemma:
\begin{lemma}\label{optimility condiyion for stability}
Denote by $u=(x,t,s,\xi,v,w)$. Then $x^{*}$ is an optimal solution of \eqref{Ps2} if and only if there exists a vector $u^{*}=(x^{*},t^{*},s^{*},\xi^{*},v^{*},w^{*})\in \Theta$, where  $\Theta$ is the set given as
\begin{equation*}\label{optimality condition for sta eq1}
\begin{array}{lll}
\Theta=\biggr\lbrace u:& -x-t\leq  0,~x-t\leq 0,~ a_{1}s+a_{2}\xi+a_{3}\left(e^{h}\right)^{T}v \leq \varepsilon,\\
& -M_{P_{0}}^{T}Ax-e^{N}s\leq  -M_{P_{0}}^{T}y,~ Bx\leq  b,\\
& -U^{T}Ax-\xi e^{h}\leq -U^{T}y,~U^{T}Ax-\xi e^{h}\leq U^{T}y ,\\
& -U^{T}Ax-v\leq -U^{T}y,~U^{T}Ax-v\leq U^{T}y ,\\
& w_{1}-w_{2}+A^{T}M_{P_{0}}w_{4}+A^{T}U(w_{5}-w_{6}+w_{7}-w_{8})-B^{T}w_{9}=0,\\
& w_{1}+w_{2}\leq e,~ -a_{1}w_{3}+(e^{N})^{T}w_{4}\leq 0,~ (t,s,\xi,v,w)\geq 0,\\
& -a_{2}w_{3}+(e^{h})^{T}(w_{5}+w_{6})\leq 0,~-a_{3}e^{h}w_{3}+w_{7}+w_{8}\leq 0,\\
& e^{T}t=-\varepsilon w_{3}+y^{T}M_{P_{0}}w_{4}+y^{T}U(w_{5}-w_{6}+w_{7}-w_{8})-b^{T}w_{9}\biggr\rbrace.
\end{array}
\end{equation*}
\end{lemma}
Clearly, $\vert x^{*}\vert=t^{*}$  holds for every $u^{*}\in \Theta$.  The set $\Theta$ can be written as the form
\begin{equation}\label{Pi set}
\Theta=\left\{ u: M'_{1}u\leq p',~M'_{2}u= q'\right\},
\end{equation}
where the vectors  $q'=0$ and
 \begin{multline*}\label{pq}
  p'=\left[\begin{array}{cccccccccccccc}
0&0&\varepsilon&-M_{P_{0}}^{T}y&b&-U^{T}y&U^{T}y&-U^{T}y&U^{T}y&e&0&0&0&0
\end{array}\right.\\
\left.\begin{array}{cccccccccccc}
0&0&0&0&0&0&0&0&0&0&0&0
\end{array}\right]^{T}.
\end{multline*}
 The matrices $M'_{1}$ and $M'_{2}$ in \eqref{Pi set}  are given as follows:
\begin{equation}\label{M_{1}}
 M_{1}'=\left[ \begin{array}{cc}
D^{1}&0  \\
0&D^{2} \\
D^{3}&0\\
0& -\widetilde{I}  \\
\end{array} \right],~M_{2}'=\left[ \begin{array}{cc}
M_{*}&M_{**} \\
\end{array} \right],
\end{equation}
\noindent  where the matrices $M_{*}$, $M_{**}$, $D^{1}$, $D^{2}$ and $D^{3}$ and $\widetilde{I}$ are given as follows:
\begin{subequations}
\begin{equation*}\label{M_{2}}
    M_{*}=\left[ \begin{array}{cccccccc}
0&0& 0& 0&0&I&-I& 0\\
0&e^{T}& 0& 0&0&0&0& \epsilon\\
\end{array} \right],
\end{equation*}
\begin{equation*}
M_{**}=\left[ \begin{array}{cccccc}
 A^{T}M_{P_{0}}&A^{T}U&-A^{T}U&A^{T}U& -A^{T}U& -B^{T}\\
 -y^{T}M_{P_{0}}&-y^{T}U&y^{T}U&-y^{T}U& y^{T}U& b^{T}\\
\end{array} \right],
\end{equation*}
\end{subequations}
$$ D^{1}=\left[ \begin{array}{ccccc}
-I&-I& 0& 0&0\\
I&-I& 0& 0&0\\
0&0& a_{1}& a_{2}&a_{3}e^{T}\\
\tiny{-M_{P_{0}}^{T}A}&0& -e^{N}& 0&0\\
B&0& 0& 0&0\\
-U^{T}A&0& 0& -e^{h}&0\\
U^{T}A&0& 0& -e^{h}&0\\
-U^{T}A&0& 0& 0&-I^{h}\\
U^{T}A&0& 0& 0&-I^{h}\\
\end{array} \right],~D^{3}=\left[ \begin{array}{ccccc}
0&-I& 0& 0&0 \\
0&0& -1& 0&0\\
0&0& 0& -1&0\\
0&0& 0& 0&-I^{h}\\
\end{array} \right], $$
$$D^{2}=\left[ \begin{array}{ccccccccc}
  I&I& 0& 0&0&0&0& 0& 0\\
  0&0& -a_{1}& (e^{N})^{T}&0&0&0& 0& 0\\
   0&0& -a_{2}& 0&(e^{h})^{T}&(e^{h})^{T}&0& 0& 0\\
  0&0& -a_{3}e^{h}& 0&0&0&I^{h}& I^{h}& 0\\
  \end{array} \right],~\widetilde{I}=I^{2n+1+N+4h+l}.$$
In the above matrices, $0$'s  are zero matrices with suitable sizes and  $I$,  $I^{h}$ and $\widetilde{I}$ are the $n\times n$, $h\times h$ and $(2n+1+N+4h+l)\times (2n+1+N+4h+l)$ identity matrices, respectively.

To prove the main stability result, we also need  the next  two Lemmas.

\begin{lemma}[Hoffman \cite{hoffman,robinson1973}]\label{hoffman theorem}
Let $M_{1}\in R^{m\times n}$ and $M_{2}\in R^{l\times n}$ be two given matrices and the set  $\mathcal{Q}$ be given as
\begin{equation*}\label{set Q}
\mathcal{Q}=\lbrace x\in R^{n}: M_{1}x\leq p, ~M_{2}x=q\rbrace.
\end{equation*}
For any vector $x\in R^{n}$, there exists a vector $x^{*}\in \mathcal{Q}$ satisfying
$$\left\|x-x^{*}\right\|_{2}\leq \sigma\left(M_{1}, M_{2}\right)\left\| \left[ \begin{array}{c}
\left(M_{1}x-p\right)^{+}\\
M_{2}x-q
\end{array} \right] \right\|_{1},$$
where $\sigma(M_{1},M_{2})$ is a constant determined by $M_{1}$ and $M_{2}$.
\end{lemma}

\noindent The constant $\sigma(M_{1},M_{2})$ is also called the    Robinson constant.  We also use the following lemma in the proof of the main  result in this section.

\begin{lemma} [\cite{zhaobook2018,zhaolistability}]\label{operator lemma}
 Let $\pi_{S}(x)$ be the projection of $x$ into the convex set $S$, i.e., $\pi_{S}(x)=\arg \min_{z\in S} \Vert x-z\Vert_{2}.$
 Let the three convex compact sets $T_{1}$, $T_{2}$ and $T_{3}$ satisfy that $T_{1}\subseteq T_{2}$ and $T_{3}\subseteq T_{2}. $ Then for any $x\in R^{n}$ and any $z\in T_{3}$ the following holds:
\begin{equation*}\label{operator eq1}
\left\|x-\pi_{T_{1}}(x)\right\|_{2}\leq \delta^{\mathcal{H}}\left(T_{1}, T_{2}\right)+2\left\| x-z\right\|_{2}.
\end{equation*}
\end{lemma}

 We also define two types of  constants. Let
\begin{equation}\label{matrix C}
C=\left[ A^{T}, B^{T}\right]^{T}=\left[ \begin{array}{c}
A\\B
\end{array} \right]
\end{equation}
be a matrix with full row rank. Given three positive numbers $c,d,\widehat{d}\in [1, \infty]$, we define the constants $\Upsilon(d,\widehat{d})$ and $\vartheta(c)$ as follows:
\begin{subequations}
\begin{equation}\label{parameter c}
     \Upsilon(d,\widehat{d})=\max_{\mho\subseteq \lbrace 1,...,h\rbrace,\vert \mho\vert=m} \left\| U_{\mho}^{-1}\right\|_{\widehat{d}\rightarrow d} \left\| \left(CC^{T}\right)^{-1}C\right\|_{\infty\rightarrow \widehat{d}},
\end{equation}
\begin{equation}\label{parameter var}
 \vartheta(c)=\left\| \left(CC^{T}\right)^{-1}C\right\|_{\infty\rightarrow c}.
\end{equation}
\end{subequations}
We will use the above constants together with the specific constants $\Upsilon(1,1), \Upsilon(\infty,\infty)$ and $\vartheta(1)$ in the stability analysis  of \eqref{Ps}. The main  result is given as follows.

\begin{theorem}\label{main stability theorem}
Let the problem data $(U, A,B,\varepsilon,a_{1},a_{2},a_{3},b,y)$ of \eqref{Ps} be given, and the matrix  $C\in R^{(m+l)\times n}$ be given in \eqref{matrix C}  with full row rank.   Let  $P_{0}$ be  the polytope given in \eqref{polytope appximated by B} satisfying \eqref{final polytope}. If $\left(A^{T}, B^{T}\right)$ satisfies the restricted weak $\mathrm{RSP}$ of order $k$, then
for any $x\in R^{n}$, there is an optimal  solution $x^{*}$ of \eqref{Ps} satisfying the  bound
\begin{equation}\label{main stability theorem eq1}
\begin{array}{lll}
\left\| x-x^{*}\right\|_{2} & \leq &  \varepsilon'+2\sigma'\biggr\lbrace 2\sigma_{k}(x)_{1}+\varepsilon \hat{\Upsilon}+\left\|(Bx-  b)^{+}\right\|_{1}+
+\left\|Bx-b\right\|_{c'}\vartheta(c)+\\
&&\left\| \phi(x) \right\|_{d'}\Upsilon(d,\widehat{d})+\left(a_{1}\left\|y-Ax\right\|_{2}+a_{2}\left\| \phi(x)\right\|_{\infty}+a_{3}\left\| \phi(x)\right\|_{1} - \varepsilon\right)^{+}
\biggr\rbrace.
\end{array}
\end{equation}
where $\sigma'$ is the Robinson constant determined by  $(M'_{1},M'_{2})$ in \eqref{M_{1}},
$\Upsilon(d,\widehat{d})$ and $\vartheta(c)$ are the constants given in \eqref{parameter c} and \eqref{parameter var}, and $\hat{\Upsilon}=\max\lbrace \Upsilon(1,1),\Upsilon(\infty,\infty), \vartheta(1)\rbrace. $  $\widehat{d},d,c,d',c'\in[1,+\infty]$ are five given  positive numbers (allowing to be $\infty$) satisfying
\begin{equation}\label{dce}
\frac{1}{c}+\frac{1}{c'}=1~\mathrm{and} ~\frac{1}{d}+\frac{1}{d'}=1.
\end{equation}
In particular, if $x$ is a feasible solution of \eqref{Ps}, then there is an optimal solution $x^{*}$ of \eqref{Ps} such that
\begin{equation}\label{main stability theorem eq2}
\begin{array}{lll}
\left\| x-x^{*}\right\|_{2}  \leq   \varepsilon'+2\sigma'\biggr\lbrace \varepsilon \hat{\Upsilon}+ 2\sigma_{k}(x)_{1}+\left\| \phi(x) \right\|_{d'}\Upsilon(d,\widehat{d})+\left\|Bx-b\right\|_{c'}\vartheta(c)
\biggr\rbrace.
\end{array}
\end{equation}
\end{theorem}

\begin{proof}
Let $x$ be any given vector in $R^{n}$ and $P_{0}$ be the fixed polytope given in \eqref{polytope appximated by B} satisfying \eqref{final polytope} in Lemma \ref{lemma approximat level}.  We  let $(t,s,\xi,v)$ satisfy that
\begin{equation}\label{main stability theorem eq4}
t=\vert x\vert,~s=\left\| (M_{P_{0}})^{T}(y-Ax)\right\|_{\infty},~
\xi=\left\| U^{T}(y-Ax)\right\|_{\infty},~ v=\left| U^{T}(y-Ax)\right|.
\end{equation}
With such a choice of $(t,s,\xi,\nu)$,  we have
\begin{equation}\label{main stability theorem eq8}
\begin{array}{lll}
\left(-x-t\right)^{+}=0,~\left(x-t\right)^{+}=0,~
\left(M_{P_{0}}^{T}(y-Ax)-e^{N}s\right)^{+}=0,\\
\left(U^{T}(y-Ax)-\xi e^{h}\right)^{+}=0,~\left(-U^{T}(y-Ax)-\xi e^{h}\right)^{+}=0,\\
\left(U^{T}(y-Ax)-v\right)^{+}=0,~\left(-U^{T}(y-Ax)-v\right)^{+}=0.
\end{array}
\end{equation}
Let $J$ be the support set of $k$ largest absolute entries of $x$, and $J_{1}$ and $J_{2}$ be the sets such that
$$J_{1}=\lbrace i:~x_{i}>0,~i\in J\rbrace,~J_{2}=\lbrace  i:~x_{i}<0,~i\in J\rbrace.$$
Clearly, $\vert J_{1}\cup J_{2}\vert=\vert J \vert=\vert J_{1}\vert+\vert J_{2}\vert\leq k$. Let $J_{3}$ be the complementary set of $J$. Clearly,  $J_{1}$, $J_{2}$ and $J_{3}$ are disjoint.  Under the assumption of  restricted weak RSP of order $k$, there exists a vector $\eta\in R\left(A^{T},B^{T}\right)$ such that $\eta=A^{T}\nu^{*}+B^{T}h^{*}$ for some $\nu^{*}\in R^{m}$ and $h^{*}\in R^{l}_{-}$ satisfying
\begin{equation}\label{main stability theorem eq5}
\eta_{i}=1~\mathrm{for}~i\in J_{1};~\eta_{i}=-1~\mathrm{for}~i\in J_{2};~\left| \eta_{i}\right|\leq 1~\mathrm{for}~ i\in J_{3}.
\end{equation}
Now we construct  a feasible solution  $w=(w_{1},...,w_{9})$ to the dual problem  \eqref{Psdual}.
\vskip 0.15in
\noindent \underline{Constructing ($w_{1},w_{2}$)}. Set $w_{1}$ and $w_{2}$ as follows:
\begin{equation*}\label{w1w2}
\left\{\begin{matrix}
(w_{1})_{i}=0,~ (w_{2})_{i}=1, & i\in J_{1};\\
 (w_{1})_{i}=1,~ (w_{2})_{i}=0, & i\in J_{2};\\
 (w_{1})_{i}=\frac{1-\eta_{i}}{2},~ (w_{2})_{i}=\frac{1+\eta_{i}}{2}, & i\in J_{3}.
\end{matrix}\right.
\end{equation*}
Such $w_{1}$ and $w_{2}$ satisfy that
\begin{equation}\label{choice w result1}
w_{1}+w_{2}\leq e,~w_{2}-w_{1}=\eta, ~w_{1},w_{2}\geq 0.
\end{equation}
\underline{Constructing ($w_{5}$--$w_{8}$)}. Note that $U$ is a matrix with  full row rank. There must exist an invertible $m\times m$ matrix of $U$, denoted by $U_{\mho}$, where $\mho\subseteq  \lbrace 1,...,h\rbrace$ with $\vert \mho\vert=m$. Denote the complementary set of $\mho$ by $\bar{\mho}=\lbrace 1,...,h\rbrace \setminus \mho.$ Then we construct a vector $g\in R^{h}$ satisfying $g_{\mho}=U^{-1}_{\mho}\nu^{*} $ and $ g_{\bar{\mho}}=0,$ which imply that
\begin{equation}\label{Mg}
Ug=\nu^{*}.
\end{equation}
Let $g^{+}$ ($g^{-}$) be the vector obtained by keeping the  positive (negative) components of $g$ and setting the remaining components to $0$. By using the vector $g$, $w_{5}$--$w_{8}$ can be constructed as follows:
\begin{equation}\label{w5678}
w_{5}=a_{2}g^{+}, ~w_{6}=-a_{2}g^{-},~ w_{7}=a_{3}g^{+},~w_{8}=-a_{3}g^{-},
\end{equation}
which implies that
\begin{equation}\label{choice w result2}
w_{5}-w_{6}+w_{7}-w_{8}=(a_{2}+a_{3})g,~w_{5},w_{6},w_{7},w_{8}\geq 0.
\end{equation}
\underline{Constructing $w_{4}$}.  Without loss of generality, we suppose that  the first $m$ columns in $M_{P_{0}}$ are $ \beta_{j},~ j=1,...,m,$ and  $ -\beta_{j}, ~j=1,...,m$ are  the second $m$ columns of  $M_{P_{0}}$. The components of $w_{4}$ can be assigned as follows:
\begin{equation*}\label{w4}
\left\{\begin{array}{cl}
 (w_{4})_{j}=a_{1}\nu^{*}_{j}, & \mathrm{if}~~\nu^{*}_{j}> 0,~j=1,...,m;\\
(w_{4})_{j+m}=-a_{1}\nu^{*}_{j}, & \mathrm{if}~~ \nu^{*}_{j}< 0,~j=1,...,m;\\
0, &\mathrm{otherwise}.
\end{array}\right.
\end{equation*}
From this choice of $w_{4}$, we can see that
\begin{equation}\label{choice w result3}
M_{P_{0}}w_{4}=a_{1}\nu^{*}, ~\left\| w_{4}\right\|_{1}=a_{1}\left\| \nu^{*}\right\|_{1}~\mathrm{and}~w_{4}\geq 0.
\end{equation}
\underline{Constructing $w_{3}$}. Let $w_{3}=\max\left\{ \left\|\nu^{*}\right\|_{1},\left\|g\right\|_{1},\left\|g\right\|_{\infty}\right\}$. Such a choice of $w_{3}$ together with the choice of  $w_{4}$--$w_{8}$ implies that
\begin{equation}\label{choice w result4}
\left\{\begin{array}{lll}
\left(-a_{1}w_{3}+(e^{N})^{T}w_{4}\right)^{+}& \leq &\left(-a_{1} \left\| \nu^{*}\right\|_{1}+(e^{N})^{T}w_{4}\right)^{+}=0,\\
\left(-a_{2}w_{3}+e^{T}(w_{5}+w_{6})\right)^{+}& \leq & \left(-a_{2}\left\| g\right\|_{1}+a_{2}\left\| g\right\|_{1}\right)^{+} =0,\\
\left(-a_{3}ew_{3}+w_{7}+w_{8}\right)^{+}&\leq &\left(-a_{3}e\left\| g\right\|_{\infty}+a_{3} \vert g\vert\right)^{+}=0.
\end{array} \right.
\end{equation}
\underline{Constructing $w_{9}$}. Let $w_{9}=-h^{*}$. Clearly,  $w_{9}\geq 0$ due to  $h^{*}\leq 0$.
\vskip 0.15in
With the above choice of  $w$, we deduce from \eqref{choice w result1}, \eqref{choice w result2}, \eqref{choice w result3} and \eqref{choice w result4} that
\begin{equation}\label{main stability theorem eq7}
\left\{\begin{array}{lll}
 w_{1}-w_{2}+A^{T}M_{P_{0}}w_{4}+A^{T}U(w_{5}-w_{6}+w_{7}-w_{8})-B^{T}w_{9}=0,\\
\left(w_{1}+w_{2}- e\right)^{+}=0,~\left(-a_{1}w_{3}+(e^{N})^{T}w_{4}\right)^{+}=0,\\
\left(-a_{2}w_{3}+e^{T}(w_{5}+w_{6})\right)^{+}=0,~\left(-a_{3}ew_{3}+w_{7}+w_{8}\right)^{+}=0,\\
t^{-}=0,~s^{-}=0,~\xi^{-}=0,~v^{-}=0,~w^{-}=0.
\end{array}\right.
\end{equation}
 Let $\mathcal{X}$ and $\mathcal{Y}$ be defined as follows:
$$\left\{
\begin{array}{l}
     \mathcal{X}=e^{T}t+\varepsilon w_{3}-y^{T}M_{P_{0}}w_{4}-y^{T}U(w_{5}-w_{6}+w_{7}-w_{8})+b^{T}w_{9},\\
\mathcal{Y}= \left(a_{1}s+a_{2}\xi+a_{3}e^{T}v - \varepsilon\right)^{+}.
\end{array}
\right.
$$
For the vector $u=(x, t,s,\xi,\nu,w)$ where  $(t,s,\xi,\nu,w)$ is constructed above, by Lemma  \ref{hoffman theorem},  there exists a vector $\hat{u}\in \Theta, $ where $\Theta$ is given in Lemma \ref{optimility condiyion for stability} and written as \eqref{Pi set}, such that
\begin{equation}\label{main stability theorem eq3}
\left\|u- \hat{u} \right\|_{2}\leq \sigma'\left\| \left[ \begin{array}{c}
\mathcal{X}\\
\mathcal{Y}\\
(Bx-  b)^{+}\\
(-x-t)^{+}\\
(x-t)^{+}\\
\left(M_{P_{0}}^{T}(y-Ax)-e^{N}s\right)^{+}  \\
\left(U^{T}(y-Ax)-\xi e^{h}\right)^{+} \\
\left(-U^{T}(y-Ax)-\xi e^{h}\right)^{+} \\
\left(U^{T}(y-Ax)-v\right)^{+}\\
\left(-U^{T}(y-Ax)-v\right)^{+}\\
\left(w_{1}+w_{2}- e\right)^{+}\\
\left(-a_{1}w_{3}+(e^{N}\right)^{T}w_{4})^{+}\\
\left(-a_{2}w_{3}+e^{T}(w_{5}+w_{6})\right)^{+}\\
\left(-a_{3}ew_{3}+w_{7}+w_{8}\right)^{+}\\
\lbrace w_{1}-w_{2}+A^{T}M_{P_{0}}w_{4}+\\A^{T}U(w_{5}-w_{6}+w_{7}-w_{8})-B^{T}w_{9}\rbrace\\
(t^{-},s^{-},\xi^{-},v^{-},w^{-})
\end{array} \right]\right\|_{1}
\end{equation}
where $\sigma'$ is the Robinson constant determined by  $(M'_{1},M'_{2})$ given by \eqref{M_{1}}.
Since the vector $(x,t,s,\xi,v,w)$ satisfies \eqref{main stability theorem eq8} and \eqref{main stability theorem eq7}, the inequality  \eqref{main stability theorem eq3} can be simplified to
\begin{equation}\label{main stability theorem eq9'}
\left\|u- \hat{u} \right\|_{2}\leq \sigma(M'_{1}, M'_{2})\lbrace \left| \mathcal{Y}\right|+\left\|(Bx-  b)^{+}\right\|_{1}+\left| \mathcal{X}\right|\rbrace.
\end{equation}
In the reminder of the proof, we estimate the terms on the right-hand side of \eqref{main stability theorem eq9'}. Note that the vectors in $T$ are unit vectors. It is easy to see that
$$\max_{1\leq i\leq N} \left|(M_{P_{0}})^{T}(Ax-y) \right|_{i}\leq \left\| y-Ax\right\|_{2}.$$ The value of  $s$ in \eqref{main stability theorem eq4} implies that $s\leq \left\| y-Ax\right\|_{2}.$
Therefore we have
\begin{equation}\label{main stability theorem eq19}
\mathcal{Y}\leq \left(a_{1}\left\| y-Ax\right\|_{2}+a_{2}\left\| U^{T}(y-Ax)\right\|_{\infty}+a_{3}\left\| U^{T}(y-Ax)\right\|_{1} - \varepsilon\right)^{+}.
\end{equation}
Due to  \eqref{Mg}, \eqref{choice w result2} and \eqref{choice w result3},  we have
\begin{equation*}
\begin{split}
\left| \mathcal{X}\right|&=\left| e^{T}t+\varepsilon w_{3}-y^{T}\nu^{*}-b^{T}h^{*}\right|\\
&=\left| e^{T}t+\varepsilon w_{3}-x^{T}A^{T}\nu^{*}+(\phi(x))^{T}g+(Bx-b)^{T}h^{*}-x^{T}B^{T}h^{*}\right|.
\end{split}
\end{equation*}
The fact  $A^{T}\nu^{*}+B^{T}h^{*}=\eta$ (due to the restricted weak RSP of order $k$)  and the  triangle inequality imply that
\begin{equation}\label{main stability theorem eq9}
\left| \mathcal{X}\right|\leq \left| e^{T}t-x^{T}\eta\right|+\varepsilon\left| w_{3}\right|+\left|(\phi(x))^{T}g\right|+\left|(Bx-b)^{T}h^{*}\right|.
\end{equation}
Now we deal with the right-hand side of the above inequality. First,
by using the index sets $J$ and $J_{3}$, we have
\begin{equation*}\label{main stability theorem eq11}
\left| e^{T}t-x^{T}\eta\right|=\left| e^{T}_{J}t_{J}+e^{T}_{J_{3}}t_{J_{3}}-x^{T}_{J}\eta_{J}-x^{T}_{J_{3}}\eta_{J_{3}}\right|.
\end{equation*}
 It follows from  $t=\vert x\vert$
 and \eqref{main stability theorem eq5}  that
\begin{eqnarray*}
\left| e^{T}_{J}t_{J}+e^{T}_{J_{3}}t_{J_{3}}-x^{T}_{J}\eta_{J}-x^{T}_{J_{3}}\eta_{J_{3}}\right| & = & \left| e^{T}_{J_{3}}t_{J_{3}}-x^{T}_{J_{3}}\eta_{J_{3}}\right| \leq  \left| e^{T}_{J_{3}}t_{J_{3}}\right| + \left| x^{T}_{J_{3}}\eta_{J_{3}}\right|\\
& \leq & \left\| x_{J_{3}}\right\|_{1}+\left| x^{T}_{J_{3}}\right| \left|\eta_{J_{3}}\right| \leq \left\| x_{J_{3}}\right\|_{1}+\left| x^{T}_{J_{3}}\right| e \\
&=&2\left\| x_{J_{3}}\right\|_{1}.
\end{eqnarray*}
 Then we  obtain
 \begin{equation}\label{main stability theorem eq12}
\left| e^{T}t-x^{T}\eta\right|\leq 2\left\| x_{J_{3}}\right\|_{1}=2\sigma_{k}(x)_{1}.
\end{equation}
\noindent By using the restricted weak $\mathrm{RSP}$ of order $k$, we have $$\left\|\nu^{*}\right\|_{1}\leq \left\|\left[ \begin{array}{c}
\nu^{*}\\h^{*}
\end{array} \right]\right\|_{1}\leq \left\| \left(CC^{T}\right)^{-1}C \eta\right\|_{1}\leq \left\| \left(CC^{T}\right)^{-1}C \right\|_{\infty\rightarrow 1} \left\|  \eta\right\|_{\infty}\leq \vartheta(1),$$
where $C=\left[ A^{T}, B^{T}\right]^{T}\in R^{(m+l)\times n}$ and $\vartheta(1)$ is defined in \eqref{parameter var}.
Moreover, we have
\begin{eqnarray*}
\left\| g\right\|_{1}= \left\| g_{\mho}\right\|_{1}  = \left\| U_{\mho}^{-1}\nu^{*}\right\|_{1}\leq \left\| U_{\mho}^{-1}\right\|_{1\rightarrow 1}\left\| \nu^{*}\right\|_{1}\leq \left\| U_{\mho}^{-1}\right\|_{1\rightarrow 1}\vartheta(1)
\end{eqnarray*}
Recall that $\Upsilon(1,1)$ is  determined in \eqref{parameter c}. Then $\left\| g\right\|_{1}\leq \Upsilon(1,1)$.
Similarly,  $\left\| g\right\|_{\infty}\leq \Upsilon(\infty,\infty)$ can be obtained.
Due to  $w_{3}=\max\left\{ \left\|\nu^{*}\right\|_{1},\left\|g\right\|_{1},\left\|g\right\|_{\infty}\right\}$, we have
\begin{equation}\label{main stability theorem eq10}
\varepsilon\left| w_{3}\right|\leq \varepsilon \max\lbrace \Upsilon(1,1),\Upsilon(\infty,\infty),\vartheta(1)\rbrace.
\end{equation}

Let $c,d, \widehat{d}\in [1,+\infty]$ be three given positive numbers and $d,d'$ be two given numbers satisfying \eqref{dce}. For the term $\left| (\phi(x))^{T}g\right|$ in \eqref{main stability theorem eq9}, it follows from  H\"older inequalities that
\begin{equation}\label{main stability theorem eq13}
\begin{array}{lll}
\left| (\phi(x))^{T}g\right| & \leq &  \left\| \phi(x)\right\|_{d'}\left\| g\right\|_{d}\\
&=&\left\| \phi(x)\right\|_{d'} \left\| U_{\mho}^{-1}\nu^{*}\right\|_{d}\\
&\leq &\left\| \phi(x)\right\|_{d'} \left\| U_{\mho}^{-1}\right\|_{\widehat{d}\rightarrow d}\left\| \nu^{*}\right\|_{\widehat{d}}\\
&\leq &\left\| \phi(x)\right\|_{d'} \left\| U_{\mho}^{-1}\right\|_{\widehat{d}\rightarrow d}\left\| \left(CC^{T}\right)^{-1}C\right\|_{\infty\rightarrow \widehat{d}}.
\end{array}
\end{equation}
Let $\Upsilon(d,\widehat{d})$ be given as \eqref{parameter c}, i.e.,
$$\Upsilon(d,\widehat{d})=\max_{\mho\subseteq \lbrace 1,...,h\rbrace,\vert \mho\vert=m} \left\| U_{\mho}^{-1}\right\|_{\widehat{d}\rightarrow d} \left\| \left(CC^{T}\right)^{-1}C\right\|_{\infty\rightarrow \widehat{d}}.$$
Thus   we have
\begin{equation}\label{main stability theorem eq14}
\left| (\phi(x))^{T}g\right| \leq \Upsilon(d,\widehat{d})\left\| \phi(x)\right\|_{d'}.
\end{equation}
Similarly, the following inequalities holds
\begin{equation}\label{main stability theorem eq17}
\begin{array}{lll}
\left| (Bx-b)^{T}h^{*}\right| & \leq &  \left\| Bx-b\right\|_{c'}\left\| h^{*}\right\|_{c}\\
&\leq&\left\| Bx-b\right\|_{c'}\left\| (CC^{T})^{-1}C\right\|_{\infty\rightarrow c}\left\|\eta\right\|_{\infty}\\
&\leq &\left\| Bx-b\right\|_{c'}\left\| (CC^{T})^{-1}C\right\|_{\infty\rightarrow c}\\
&=&\vartheta(c)\left\| Bx-b\right\|_{c'}.
\end{array}
\end{equation}
Due to  \eqref{main stability theorem eq12}, \eqref{main stability theorem eq10}, \eqref{main stability theorem eq14} and \eqref{main stability theorem eq17}, the inequality \eqref{main stability theorem eq9} is reduced to
 \begin{equation}\label{main stability theorem eq18}
\left| \mathcal{X}\right|  \leq  \varepsilon \hat{\Upsilon}+2\sigma_{k}(x)_{1} +\left\| \phi(x) \right\|_{d'}\Upsilon(d,\widehat{d})+\left\|Bx-b\right\|_{c'}\vartheta(c),
 \end{equation}
 where $\hat{\Upsilon}=\max\lbrace \Upsilon(1,1),\Upsilon(\infty,\infty), \vartheta(1)\rbrace$.

 Note that $\left\| x-\hat{x}\right\|_{2}\leq \left\| u-\hat{u} \right\|_{2}.$
It follows from   \eqref{main stability theorem eq9'},  \eqref{main stability theorem eq19} and \eqref{main stability theorem eq18} that
\begin{equation}\label{main stability theorem eq15}
\begin{array}{lll}
\left\| x-\hat{x}\right\|_{2}& \leq & \sigma'\biggr\lbrace 2\sigma_{k}(x)_{1}+\left\|(Bx-  b)^{+}\right\|_{1}+\varepsilon \hat{\Upsilon}+\left\| \phi(x) \right\|_{d'}\Upsilon(d,\widehat{d}) +\\
&& \left(a_{1}\left\|y-Ax\right\|_{2}+a_{2}\left\| \phi(x)\right\|_{\infty}+a_{3}\left\| \phi(x)\right\|_{1} - \varepsilon\right)^{+}+\left\|Bx-b\right\|_{c'}\vartheta(c)\biggr\rbrace.
\end{array}
\end{equation}
We recall the three sets $\Omega^{*}$, $\Omega_{P_{0}}$ and $\Omega^{*}_{P_{0}}$, where $\Omega^{*}$ and $\Omega^{*}_{P_{0}}$ are the solution sets of \eqref{Ps} and \eqref{Ps2}, given as \eqref{Omega} and \eqref{solution set of appro}, respectively,  and $\Omega_{P_{0}}$ is given as \eqref{Omega1} with $P=P_{0}$. Clearly,  $\hat{x}\in \Omega^{*}_{P_{0}}$. Let $x^{*}$ denote the projection of $x$ onto $\Omega^{*}$, that is, $$x^{*}=\pi_{\Omega^{*}}(x).$$ Note that the three sets are compact convex sets satisfying $\Omega^{*}\subseteq\Omega_{P_{0}}$ and $\Omega_{P_{0}}^{*}\subseteq\Omega_{P_{0}}$. Then by applying Lemma \ref{operator lemma} with $T_{1}=\Omega^{*}$, $T_{2}=\Omega_{P_{0}}$ and $T_{3}=\Omega^{*}_{P_{0}}$, we have
$$\left\|x-\pi_{\Omega^{*}}(x)\right\|_{2}=\left\|x-x^{*}\right\|_{2}\leq \delta^{\mathcal{H}}(\Omega^{*}, \Omega_{P_{0}})+2\left\| x-\hat{x}\right\|_{2}.$$
Since $P_{0}$ satisfies \eqref{final polytope}, it implies that
$$\left\|x-x^{*}\right\|_{2}\leq \varepsilon'+2\left\| x-\hat{x}\right\|_{2}.$$
Let $\hat{\Upsilon}=\max\lbrace \Upsilon(1,1),\Upsilon(\infty,\infty), \vartheta(1)\rbrace.$
Combination of the above inequality and \eqref{main stability theorem eq15} yields the desired results \eqref{main stability theorem eq1}. If $x$ is the feasible solution of \eqref{Ps}, then  $\left\|(Bx-  b)^{+}\right\|_{1}=0$ and $$\left(a_{1}\left\|y-Ax\right\|_{2}+a_{2}\left\| \phi(x)\right\|_{\infty}+a_{3}\left\| \phi(x)\right\|_{1} - \varepsilon\right)^{+}=0,$$ and thus the desired error bound  \eqref{main stability theorem eq2} is also  obtained.
\end{proof}
Based on Theorem \ref{main stability theorem}, the error bound for the solutions of \eqref{Ps0} and \eqref{Ps} can be stated as follows.

\begin{corollary}
For any optimal solution $x$ of \eqref{Ps0},  there is an optimal solution $x^{*}$ of \eqref{Ps}  estimating $x$ with the error:
\begin{equation*}\label{error bound bewteen l0 l1}
\begin{array}{lll}
\left\| x-x^{*}\right\|_{2}  \leq   \varepsilon'+2\sigma'\biggr\lbrace \varepsilon \hat{\Upsilon}+ \sigma_{k}(x)_{1}+\left\| \phi(x) \right\|_{d'}\Upsilon(d,\widehat{d})+\left\|Bx-b\right\|_{c'}\vartheta(c)
\biggr\rbrace,
\end{array}
\end{equation*}
where the constants $\varepsilon'$, $\hat{\Upsilon}$, $\sigma'$, $\Upsilon(d,\widehat{d})$ and $\vartheta(c)$ are given  as in Theorem \ref{main stability theorem}.
\end{corollary}

\section{Special cases}\label{section special case}
Firstly, by setting different values of $a_{1},a_{2}$ and $a_{3}$, the problem $\eqref{Ps}$ can   reduce to several special cases, and  the corresponding stability results for these special cases can be obtained  from \eqref{main stability theorem eq1} and \eqref{main stability theorem eq2} immediately. Note that if any of $a_{1},a_{2}$ and $a_{3}$ is zero, the constant $\hat{\Upsilon}=\max\lbrace \Upsilon(1,1),\Upsilon(\infty,\infty), \vartheta(1)\rbrace$ in \eqref{main stability theorem eq1} and \eqref{main stability theorem eq2} will be simplified as well. For example, if $a_{1}=0$, the constant $\hat{\Upsilon}$ is reduced to $\max\lbrace \Upsilon(1,1),\Upsilon(\infty,\infty)\rbrace$. The following table  shows the form of the constant  $\hat{\Upsilon}$  for different choices  of $a_{1},a_{2}$ and $a_{3}$.
\begin{table}[htp]
\tbl{The constant $\hat{\Upsilon}$.}
{\begin{tabular}{lc}\toprule
\multicolumn{1}{c}{$a_{i}$} & \multicolumn{1}{c}{$\hat{\Upsilon}$}
\\ \midrule
$a_{1}+a_{2}=0$ & $\Upsilon(\infty,\infty)$ \\
$a_{1}+a_{3}=0$ &    $\Upsilon(1,1)$\\
 $a_{2}+a_{3}=0$         &  $\vartheta(1)$\\
 $a_{1}=0$   & $\max\lbrace \Upsilon(1,1),\Upsilon(\infty,\infty)\rbrace$ \\
$a_{2}=0$ & $\max\lbrace \Upsilon(\infty,\infty), \vartheta(1)\rbrace$\\
$a_{3}=0$& $\max\lbrace \Upsilon(1,1),\vartheta(1)\rbrace$\\
$a_{1},a_{2},a_{3}\neq 0$& $\max\lbrace \Upsilon(1,1),\Upsilon(\infty,\infty), \vartheta(1)\rbrace$             \\ \bottomrule
\end{tabular}}
\label{stability table 1}
\end{table}
\noindent Note that for any case with $a_{1}=0$, we have $\Omega^{*}= \Omega_{P_{0}}=\Omega_{P_{0}}^{*}$ so that $\hat{x}=x^{*}$ where $\hat{x}\in \Omega_{P_{0}}^{*}$ and $x^{*}\in \Omega^{*}$. Thus instead of using Lemma \ref{operator lemma}, the stability results can be immediately obtained from \eqref{main stability theorem eq15}.

Secondly, without matrix $B$, the problem \eqref{Ps} is reduced to
\begin{equation*}\label{Ps5}
\begin{array}{lcl}
&\min\limits_{x\in R^{n}}&\left \| x \right \|_{1}\\
& $s.t$.& a_{1}\left \| y-Ax \right \|_{2}+a_{2}\left\| \phi(x)\right\|_{\infty}+a_{3}\left\| \phi(x)\right\|_{1} \leq \varepsilon.
\end{array}
\end{equation*}
In this case,  the restricted weak $\mathrm{RSP}$ of order $k$ is reduced to the standard weak $\mathrm{RSP}$ of order $k$, which means $A^{T}\nu^{*}=\eta$. In fact, the upper bound of $\left| (\phi(x))^{T}g\right|$ in \eqref{main stability theorem eq13} can be improved to
\begin{equation*}\label{main stability theorem eq20}
\begin{array}{lll}
\left| (\phi(x))^{T}g\right| & \leq &  \left\| \phi(x)\right\|_{d'}\left\| g\right\|_{d}\\
&=&\left\| \phi(x)\right\|_{d'} \left\| U_{\mho}^{-1}\nu^{*}\right\|_{d}\\
&\leq &\left\| \phi(x)\right\|_{d'} \left\| U_{\mho}^{-1}(AA^{T})^{-1}A\eta\right\|_{d}\\
&\leq &\left\| \phi(x)\right\|_{d'} \left\| U_{\mho}^{-1}(AA^{T})^{-1}A\right\|_{\infty\rightarrow d}.
\end{array}
\end{equation*}
Then in order to obtain a tighter bound, $\Upsilon(d,\widehat{d})$ can be replaced by $$\Upsilon'(d)=\max\limits_{\mho\subseteq \lbrace 1,...,h\rbrace, \vert \mho\vert=m} \left\| U_{\mho}^{-1}\left(AA^{T}\right)^{-1}A\right\|_{\infty\rightarrow d}.$$
Thus we have $\left| (\phi(x))^{T}g\right|\leq \left\| \phi(x)\right\|_{d'}\Upsilon'(d)  $. Similarly, the constants $\Upsilon(1,1)$ and $\Upsilon(\infty, \infty)$ are replaced by $\Upsilon'(1)$ and $\Upsilon'(\infty)$, respectively. Clearly, in this case,  $\vartheta(c)=\left\| (AA^{T})^{-1}A\right\|_{\infty\rightarrow c}$. Let $\hat{\Upsilon'}=\max\lbrace \Upsilon'(1),\Upsilon'(\infty), \vartheta(1)\rbrace$.
Then the bound \eqref{main stability theorem eq2} is reduced to
\begin{equation*}\label{main stability theorem eq21}
\begin{array}{lll}
\left\| x-x^{*}\right\|_{2} & \leq &  \varepsilon'+2\sigma'\biggr\lbrace \varepsilon \hat{\Upsilon'}+ 2\sigma_{k}(x)_{1}+\left\| \phi(x) \right\|_{d'}\Upsilon'(d)
\biggr\rbrace.
\end{array}
\end{equation*}
Similarly, we list  the constants $\hat{\Upsilon}'$ for different choices of $a_{i},~i=1,2,3$ in the following table.
\begin{table}[htp]
\tbl{The constant $\hat{\Upsilon}'$.}
{\begin{tabular}{lc}\toprule
\multicolumn{1}{c}{$a_{i}$} & \multicolumn{1}{c}{$\hat{\Upsilon}'$}
\\ \midrule
$a_{1}+a_{2}=0$ & $\Upsilon'(\infty)$ \\
$a_{1}+a_{3}=0$ &    $\Upsilon'(1)$\\
 $a_{2}+a_{3}=0$         &  $\vartheta(1)$\\
 $a_{1}=0$   & $\max\lbrace \Upsilon'(1),\Upsilon'(\infty)\rbrace$ \\
$a_{2}=0$ & $\max\lbrace \Upsilon'(\infty), \vartheta(1)\rbrace$\\
$a_{3}=0$& $\max\lbrace \Upsilon'(1),\vartheta(1)\rbrace$\\
$a_{1},a_{2},a_{3}\neq 0$& $\max\lbrace \Upsilon'(1),\Upsilon'(\infty), \vartheta(1)\rbrace$             \\ \bottomrule
\end{tabular}}
\label{stability table 2}
\end{table}
Note that when $a_{1}=0$, we have $\hat{\Upsilon}'=\Upsilon'(1)$ due to the fact
$\left\| U_{\mho}^{-1}\left(AA^{T}\right)^{-1}A\right\|_{\infty\rightarrow 1}\geq \left\| U_{\mho}^{-1}\left(AA^{T}\right)^{-1}A\right\|_{\infty\rightarrow \infty}$. Moreover, in this case, setting $d=1$ yields
\begin{equation*}\label{main stability theorem eq22}
\begin{array}{lll}
\left\| x-x^{*}\right\|_{2} & \leq &  \sigma'\biggr\lbrace \varepsilon\Upsilon'(1) + 2\sigma_{k}(x)_{1}+\left\| \phi(x) \right\|_{\infty}\Upsilon'(1)
\biggr\rbrace,
\end{array}
\end{equation*}
which is  the bound for the following $\ell_{1}$-minimization established by Zhao and Li \cite{zhaolistability} (see also in   Zhao \cite{zhaobook2018}): $$\min\lbrace \Vert x\Vert_{1}:~a_{2}\left\| \phi(x)\right\|_{\infty}+a_{3}\left\| \phi(x)\right\|_{1} \leq \varepsilon\rbrace.$$

Last but not least, our analysis can also apply to 1-bit basis pursuit \cite{zhaochun2016}, which can be  viewed as a special case of our model \eqref{Ps}. The stability result for  the 1-bit basis pursuit in \cite{zhaochun2016} can be obtained immediately from Theorem \ref{main stability theorem} by setting $a_{2}=a_{3}=0$.

\section{Conclusion}
 In this paper,    we have studied the stability issue of the $\ell_{1}$-minimization method \eqref{Ps}. To establish our  results, we  introduced   the restricted  weak RSP of order $k$ which is a mild assumption governing the stability of sparsity-seeking algorithms. Under this assumption,  we use the classic Hoffman theorem and  Lemma \ref{operator lemma} to show that  the $\ell_{1}$-minimization method \eqref{Ps} is stable and thus the error  between the solutions of the problems \eqref{Ps0} and \eqref{Ps}  can be measured in terms of the best $k$ term approximation and the problem data
(see  Theorem \ref{main stability theorem}). The result developed in this paper can apply to a  range of problems with  constraints defined by $\ell_{1}$-, $\ell_{2}$-, and $\ell_{\infty}$-norms.

\end{document}